\documentclass[11pt,a4paper]{article}
\usepackage[T1]{fontenc}
\usepackage{lmodern}
\usepackage{microtype}
\usepackage[a4paper,margin=28mm]{geometry}
\usepackage{amsmath,amssymb,amsthm,mathtools}
\usepackage{tikz}
\usepackage{aliascnt}
\usepackage{fancyhdr}

\newtheorem{theorem}{Theorem}[section]
\newaliascnt{lemma}{theorem}
\newtheorem{lemma}[lemma]{Lemma}
\aliascntresetthe{lemma}
\newaliascnt{corollary}{theorem}
\newtheorem{corollary}[corollary]{Corollary}
\aliascntresetthe{corollary}
\theoremstyle{remark}
\newtheorem{remark}[theorem]{Remark}
\theoremstyle{plain}
\usepackage[hidelinks]{hyperref}
\usepackage[nameinlink,capitalise]{cleveref}
\newcommand{\PP}{\mathbf P}
\newcommand{\RR}{\mathbf R}
\newcommand{\CC}{\mathbf C}
\newcommand{\cO}{\mathcal O}
\newcommand{\conv}{\operatorname{conv}}
\newcommand{\vol}{\operatorname{vol}}
\newcommand{\MV}{\operatorname{MV}}
\newcommand{\Dval}{\Delta^{\mathrm{val}}}
\newcommand{\authoremail}{\texttt{ylulu0610@gmail.com}}
\fancypagestyle{firstpage}{%
  \fancyhf{}%
  \fancyfoot[L]{\footnotesize\textsuperscript{*}E-mail address: \authoremail}%
  \fancyfoot[C]{\thepage}%
}
\title{Minkowski Decompositions for Generic Infinitesimal Newton–Okounkov Bodies of Arbitrary-Degree External Tensor Products on Products of Curves}
\author{Yi Lu\textsuperscript{*}}
\date{}
\begin{document}
\maketitle
\thispagestyle{firstpage}

\begin{abstract}
Explicit computations of generic infinitesimal Newton--Okounkov bodies
are difficult even for varieties with simple product structure. We give
an explicit formula in arbitrary dimension for positive-degree external
tensor products on products of smooth projective curves. Writing $d^\downarrow=(d_1^\downarrow,\ldots,d_n^\downarrow)$ for the
decreasing rearrangement of the degree vector and setting
$d_{n+1}^\downarrow=0$, the body admits the explicit Minkowski decomposition
$\sum_{j=1}^n(d_j^\downarrow-d_{j+1}^\downarrow)S_j^{(n)}$, where the
$S_j^{(n)}$ are the embedded simplices defined below.
Using this description, we give a sharp criterion for equality in the
Minkowski inclusion. A simultaneous relabeling argument also allows finitely
many such bodies to be realized on a common very general locus of flags after
independent decreasing rearrangements of their degree vectors.
\end{abstract}

\noindent\textbf{Keywords.} Newton--Okounkov body; product of curves;
Minkowski sum; mixed volume; Minkowski additivity.

\section{Introduction}
\label{sec:intro}

For products of curves, Fulger--Lozovanu computed the generic
infinitesimal body for equal factor degrees in every dimension, and for
arbitrary positive degrees in dimensions two and three
\cite[Theorems~1.1--1.2]{FLCurves}. Their Problem~6.1 asks for the
arbitrary-degree formula in general dimension. We resolve that problem for
positive-degree external tensor products by expressing the generic body as
an explicit Minkowski sum of the equal-degree simplices. The principal
contribution is therefore an explicit arbitrary-degree formula valid in all
dimensions; the simultaneous realization statement is a compatibility
refinement. In \cref{sec:additivity} we use the formula to characterize
equality in the standard Minkowski inclusion.

Throughout, the ground field is $\CC$, all curves are smooth, connected,
and projective, and
\[
 X=C_1\times\cdots\times C_n,\qquad
 L=L_1\boxtimes\cdots\boxtimes L_n,\qquad
 d=(d_1,\ldots,d_n),\quad d_i=\deg L_i>0.
\]
Fix $x=(x_1,\ldots,x_n)\in X$. For a finite-dimensional vector space $V$,
let $\operatorname{Flag}(V)$ denote its variety of complete linear flags. We
identify points of $\operatorname{Flag}(T_xX)$ with the associated
infinitesimal linear flags over $x$ described in \cref{sec:prelim}. A
\emph{very general locus} means the complement of a countable union of proper
Zariski-closed subsets. For a big line bundle $B$ on $X$ and such a flag
$Y_\bullet$, write $\Delta_{Y_\bullet}(B)$ for the corresponding infinitesimal
Newton--Okounkov body and $\Delta_x(B)$ for its common value on a very general
locus; the precise construction is recalled in \cref{sec:prelim}.

Let $e_1,\ldots,e_n$ be the standard basis of $\RR^n$, and define
\begin{equation}
 \label{eq:simplices}
 S_j^{(n)}=\conv\bigl(\{0,je_1\}\cup
 \{re_1+re_{n-j+r+1}:1\leq r<j\}\bigr),
 \qquad 1\leq j\leq n.
\end{equation}
For a positive vector $b=(b_1,\ldots,b_n)$, let $b^\downarrow$ denote its
decreasing rearrangement and set
\begin{equation}
 \label{eq:Pdef}
 \mathcal P(b):=\sum_{j=1}^n
 (b_j^\downarrow-b_{j+1}^\downarrow)S_j^{(n)},
 \qquad b_{n+1}^\downarrow=0.
\end{equation}
Let $\mathfrak S_n$ denote the permutation group of $\{1,\ldots,n\}$.
For $\sigma\in\mathfrak S_n$, let
\[
 \tau_\sigma:X\longrightarrow
 X^\sigma:=C_{\sigma(1)}\times\cdots\times C_{\sigma(n)},\qquad
 (x_i)_i\longmapsto(x_{\sigma(i)})_i,
\]
and write $x^\sigma=\tau_\sigma(x)$. For a line bundle $B$ on $X$, put
$B^\sigma=(\tau_\sigma^{-1})^*B$; for an infinitesimal flag
$Y_\bullet$ over $x$, write $Y_\bullet^\sigma$ for its transported flag
over $x^\sigma$.

\begin{theorem}[Minkowski formula and simultaneous realization]
\label{thm:bodyformula}
For every $x\in X$,
\begin{equation}
 \label{eq:main}
 \boxed{\displaystyle \Delta_x(L)=\mathcal P(d).}
\end{equation}
Under the temporary convention $d=d^\downarrow$, set $d_{n+1}=0$; then
\begin{equation}
 \label{eq:main-decreasing}
 \Delta_x(L)=\sum_{j=1}^n(d_j-d_{j+1})S_j^{(n)}.
\end{equation}
More generally, let $L^{(1)},\ldots,L^{(q)}$ be finitely many
positive-degree external products on the same $X$, with degree vectors
$d^{(\alpha)}$. For each $\alpha$, choose independently $\sigma_\alpha\in\mathfrak S_n$ which puts
$d^{(\alpha)}$ in decreasing order, and set
$L^{(\alpha),\sigma_\alpha}:=(L^{(\alpha)})^{\sigma_\alpha}$ on
$X^{\sigma_\alpha}$. Then there exists a very general,
hence nonempty, locus $G\subset\operatorname{Flag}(T_xX)$ such that for every
$Y_\bullet\in G$ and every $\alpha$,
\begin{equation}
 \label{eq:main-simultaneous}
 \Delta_{Y_\bullet}(L^{(\alpha)})
 =\Delta_x(L^{(\alpha)})
 =\Delta_{Y_\bullet^{\sigma_\alpha}}
   (L^{(\alpha),\sigma_\alpha})
 =\mathcal P(d^{(\alpha)}).
\end{equation}
The permutations $\sigma_\alpha$ are independent; no common decreasing
factor order is assumed.
\end{theorem}

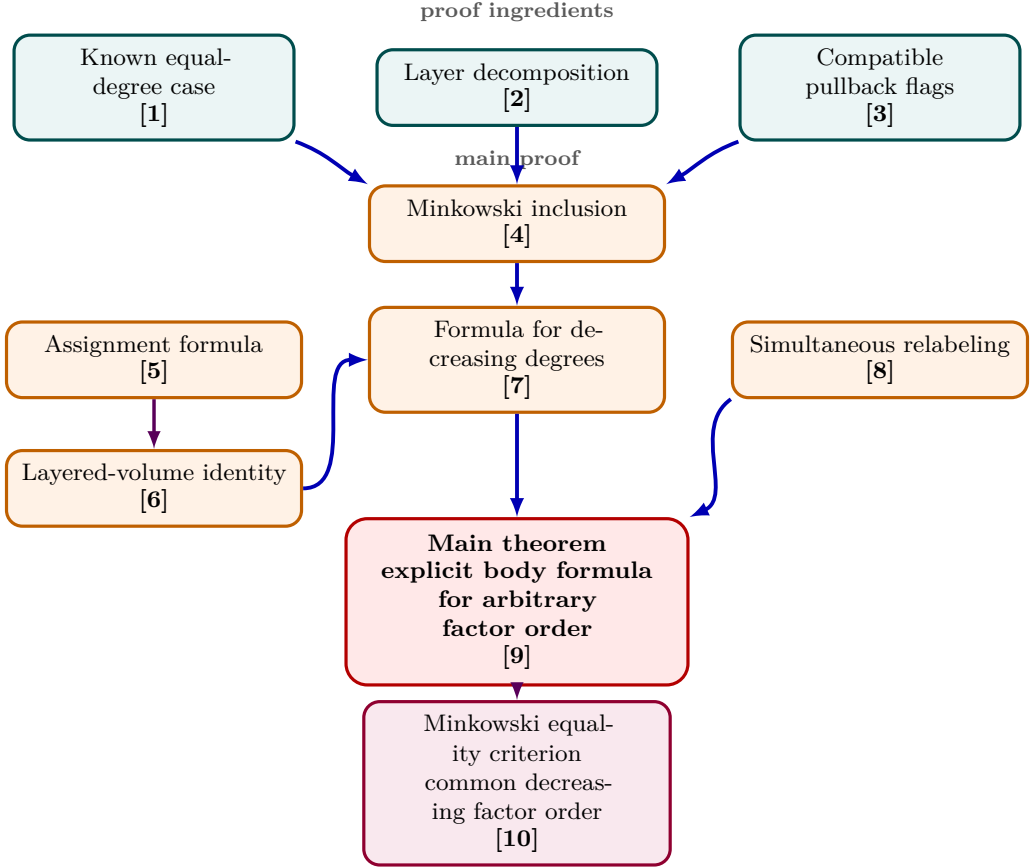
\begin{figure}[t]
\centering
\begin{tikzpicture}[
  >=latex,
  mainarrow/.style={->, line width=1.45pt, draw=blue!70!black},
  sidearrow/.style={->, line width=1.3pt, draw=violet!70!black},
  inputnode/.style={draw=teal!60!black, fill=teal!8, rounded corners=6pt,
    very thick, align=center, font=\footnotesize, inner sep=5pt, text width=3.35cm},
  methodnode/.style={draw=orange!75!black, fill=orange!10, rounded corners=6pt,
    very thick, align=center, font=\footnotesize, inner sep=5pt, text width=3.55cm},
  theoremnode/.style={draw=red!70!black, fill=red!9, rounded corners=8pt,
    very thick, align=center, font=\footnotesize\bfseries, inner sep=6pt, text width=4.1cm},
  cornode/.style={draw=purple!75!black, fill=purple!9, rounded corners=6pt,
    very thick, align=center, font=\footnotesize, inner sep=5pt, text width=3.7cm},
  stage/.style={font=\scriptsize\bfseries, text=gray!70!black}]

  \node[stage] at (0,6.15) {proof ingredients};
  \node[stage] at (0,4.2) {main proof};
  \node[stage] at (0,-0.95) {main conclusion};
  \node[stage] at (0,-4.15) {application};

  \node[inputnode] (A) at (-4.8,5.15) {Known equal-degree case\\[-1pt]
    \hyperref[eq:equal-degree]{\textbf{[1]}}};
  \node[inputnode] (B) at (0,5.15) {Layer decomposition\\[-1pt]
    \hyperref[eq:layer-decomposition]{\textbf{[2]}}};
  \node[inputnode] (C) at (4.8,5.15) {Compatible pullback flags\\[-1pt]
    \hyperref[eq:pullback-body]{\textbf{[3]}}};

  \node[methodnode] (D) at (0,3.35) {Minkowski inclusion\\[-1pt]
    \hyperref[eq:main-inclusion]{\textbf{[4]}}};

  \node[methodnode] (E) at (-4.8,1.55) {Assignment formula\\[-1pt]
    \hyperref[eq:assignment]{\textbf{[5]}}};
  \node[methodnode] (F) at (-4.8,-0.15) {Layered-volume identity\\[-1pt]
    \hyperref[lem:volume]{\textbf{[6]}}};

  \node[methodnode] (G) at (0,1.55) {Formula for decreasing degrees\\[-1pt]
    \hyperref[eq:main-decreasing]{\textbf{[7]}}};
  \node[methodnode] (H) at (4.8,1.55) {Simultaneous relabeling\\[-1pt]
    \hyperref[lem:simultaneous-relabeling]{\textbf{[8]}}};

  \node[theoremnode] (I) at (0,-1.65) {Main theorem\\[-1pt]
    explicit body formula\\
    for arbitrary factor order\\[-1pt]
    \hyperref[thm:bodyformula]{\textbf{[9]}}};

  \node[cornode] (K) at (0,-4.05) {Minkowski equality criterion\\common decreasing factor order\\[-1pt]
    \hyperref[cor:additivity]{\textbf{[10]}}};

  \draw[mainarrow] (A.south east) to[out=-20,in=150] (D.north west);
  \draw[mainarrow] (B) -- (D);
  \draw[mainarrow] (C.south west) to[out=-160,in=30] (D.north east);

  \draw[mainarrow] (D) -- (G);
  \draw[sidearrow] (E) -- (F);
  \draw[mainarrow] (F.east) to[out=0,in=180] (G.west);
  \draw[mainarrow] (G) -- (I);
  \draw[mainarrow] (H.south west) to[out=-145,in=25] (I.north east);

  \draw[sidearrow] (I) -- (K);
\end{tikzpicture}
\caption{Macro-level summary of the proof structure. Three proof ingredients first
combine to produce the Minkowski inclusion. A separate volume-computation branch
and the decreasing-order formula then feed into the main theorem, while
simultaneous relabeling removes the temporary ordering convention. The final
node records the main application in \cref{sec:additivity}. Each bracketed
number is a clickable link to the corresponding place in the text.}
\label{fig:proof-roadmap}
\end{figure}

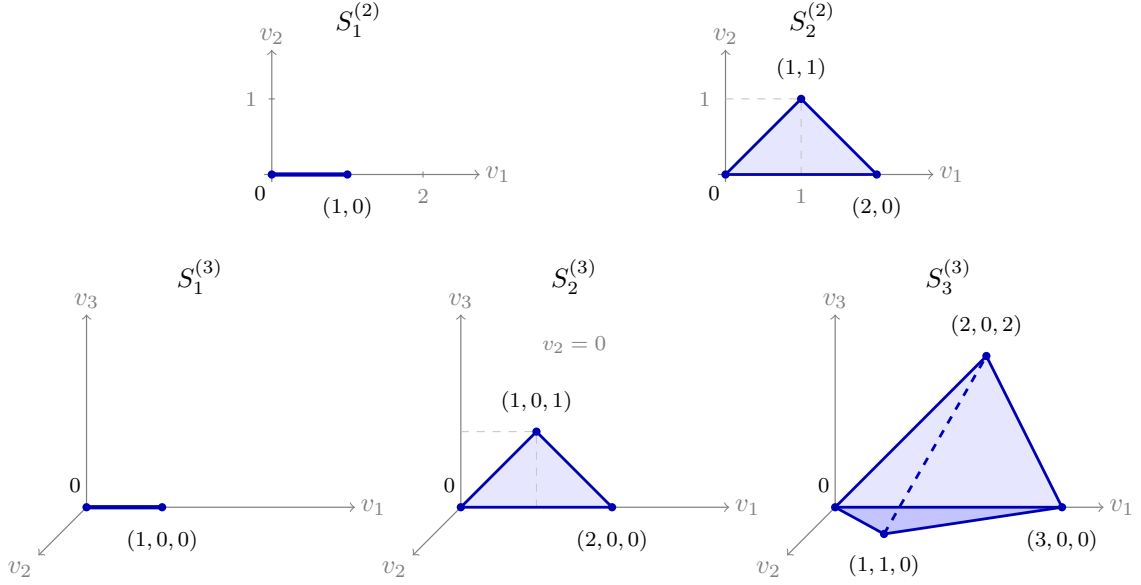
\begin{figure}[t]
\centering
\begin{tikzpicture}[
  font=\small,
  axis/.style={->,gray,line width=.45pt},
  guide/.style={gray!45,dashed,line width=.4pt},
  edge/.style={blue!70!black,line width=1.05pt},
  vertex/.style={circle,fill=blue!70!black,inner sep=0pt,minimum size=3.1pt},
  every node/.style={inner sep=2pt}]
\begin{scope}[shift={(3,4.4)}]
  \node at (1.15,2.05) {$S_1^{(2)}$};
  \draw[axis] (-.1,0)--(2.75,0) node[right] {$v_1$};
  \draw[axis] (0,-.1)--(0,1.65) node[above] {$v_2$};
  \draw[gray,line width=.4pt] (2,-.04)--(2,.04);
  \node[gray,below,font=\scriptsize] at (2,-.12) {$2$};
  \draw[gray,line width=.4pt] (-.04,1)--(.04,1);
  \node[gray,left,font=\scriptsize] at (-.12,1) {$1$};
  \draw[edge,line width=1.6pt] (0,0)--(1,0);
  \node[vertex] at (0,0) {};
  \node[vertex] at (1,0) {};
  \node[below left,font=\scriptsize] at (0,-.08) {$0$};
  \node[below,font=\scriptsize] at (1,-.22) {$(1,0)$};
\end{scope}
\begin{scope}[shift={(9,4.4)}]
  \node at (1.15,2.05) {$S_2^{(2)}$};
  \draw[axis] (-.1,0)--(2.75,0) node[right] {$v_1$};
  \draw[axis] (0,-.1)--(0,1.65) node[above] {$v_2$};
  \fill[blue!10] (0,0)--(2,0)--(1,1)--cycle;
  \draw[guide] (0,1)--(1,1)--(1,0);
  \node[gray,left,font=\scriptsize] at (-.12,1) {$1$};
  \node[gray,below,font=\scriptsize] at (1,-.12) {$1$};
  \draw[edge] (0,0)--(2,0)--(1,1)--cycle;
  \foreach \p in {(0,0),(2,0),(1,1)} \node[vertex] at \p {};
  \node[below left,font=\scriptsize] at (0,-.08) {$0$};
  \node[below,font=\scriptsize] at (2,-.22) {$(2,0)$};
  \node[above,font=\scriptsize] at (1,1.18) {$(1,1)$};
\end{scope}
\begin{scope}[shift={(.55,0)},
  x={(1cm,0cm)},y={(-.3535533906cm,-.3535533906cm)},z={(0cm,1cm)}]
  \node at (1.5,0,3.05) {$S_1^{(3)}$};
  \draw[axis] (0,0,0)--(3.55,0,0) node[right] {$v_1$};
  \draw[axis] (0,0,0)--(0,1.8,0) node[below left] {$v_2$};
  \draw[axis] (0,0,0)--(0,0,2.55) node[above] {$v_3$};
  \draw[edge,line width=1.6pt] (0,0,0)--(1,0,0);
  \node[vertex] at (0,0,0) {};
  \node[vertex] at (1,0,0) {};
  \node[above left,font=\scriptsize] at (0,0,.12) {$0$};
  \node[below,font=\scriptsize] at (1,0,-.22) {$(1,0,0)$};
\end{scope}
\begin{scope}[shift={(5.50,0)},
  x={(1cm,0cm)},y={(-.3535533906cm,-.3535533906cm)},z={(0cm,1cm)}]
  \node at (1.5,0,3.05) {$S_2^{(3)}$};
  \draw[axis] (0,0,0)--(3.55,0,0) node[right] {$v_1$};
  \draw[axis] (0,0,0)--(0,1.8,0) node[below left] {$v_2$};
  \draw[axis] (0,0,0)--(0,0,2.55) node[above] {$v_3$};
  \fill[blue!10] (0,0,0)--(2,0,0)--(1,0,1)--cycle;
  \draw[guide] (0,0,1)--(1,0,1)--(1,0,0);
  \draw[edge] (0,0,0)--(2,0,0)--(1,0,1)--cycle;
  \foreach \p in {(0,0,0),(2,0,0),(1,0,1)} \node[vertex] at \p {};
  \node[above left,font=\scriptsize] at (0,0,.12) {$0$};
  \node[below,font=\scriptsize] at (2,0,-.22) {$(2,0,0)$};
  \node[above,font=\scriptsize] at (1,0,1.18) {$(1,0,1)$};
  \node[gray,font=\scriptsize] at (1.5,0,2.15) {$v_2=0$};
\end{scope}
\begin{scope}[shift={(10.45,0)},
  x={(1cm,0cm)},y={(-.3535533906cm,-.3535533906cm)},z={(0cm,1cm)}]
  \node at (1.5,0,3.05) {$S_3^{(3)}$};
  \draw[axis] (0,0,0)--(3.55,0,0) node[right] {$v_1$};
  \draw[axis] (0,0,0)--(0,1.8,0) node[below left] {$v_2$};
  \draw[axis] (0,0,0)--(0,0,2.55) node[above] {$v_3$};
  \fill[blue!10] (0,0,0)--(3,0,0)--(2,0,2)--cycle;
  \fill[blue!23] (0,0,0)--(3,0,0)--(1,1,0)--cycle;
  \draw[edge] (0,0,0)--(3,0,0)--(2,0,2)--cycle;
  \draw[edge] (0,0,0)--(1,1,0)--(3,0,0);
  \draw[edge,dashed] (1,1,0)--(2,0,2);
  \foreach \p in {(0,0,0),(3,0,0),(1,1,0),(2,0,2)} \node[vertex] at \p {};
  \node[above left,font=\scriptsize] at (0,0,.12) {$0$};
  \node[below,font=\scriptsize] at (3,0,-.22) {$(3,0,0)$};
  \node[below,font=\scriptsize] at (1,1,-.20) {$(1,1,0)$};
  \node[above,font=\scriptsize] at (2,0,2.18) {$(2,0,2)$};
\end{scope}
\end{tikzpicture}
\caption{Examples of the embedded simplices $S_j^{(n)}$: $n=2$, $j=1,2$
in the top row, and $n=3$, $j=1,2,3$ in the bottom row. The
three-dimensional panels use a cabinet projection: $v_1$ is horizontal,
$v_3$ is vertical, and the diagonal $v_2$-axis is drawn at $45^\circ$
with half scale. The dashed blue edge of the tetrahedron is hidden.
All vertex labels give the original coordinates.}
\label{fig:simplices-dim2-dim3}
\end{figure}

\Cref{fig:simplices-dim2-dim3} illustrates the simplices in low dimension,
and \cref{fig:proof-roadmap} summarizes the dependency structure of the
proof. The decomposition, equal-degree input, and compatible pullback flags
give the Minkowski inclusion; the layered-volume identity supplies the volume
comparison; simultaneous relabeling then removes the temporary ordering
assumption. The only point where valuative, rather than big,
Newton--Okounkov bodies are needed is the pullback step for bundles from
proper subproducts.

For ordinary Newton--Okounkov bodies, Wilms proved additivity on
two-dimensional subcones for suitable flags \cite{WilmsAdditivity}. This
highlights the role of the flag in additivity questions. In
\cref{sec:additivity} the explicit formula gives a sharp criterion for equality
in the Minkowski inclusion for generic infinitesimal bodies realized on a common
very general locus of flags.

The broader computation for classes involving diagonal directions is posed
separately in \cite[Problem~6.2]{FLCurves}.

\section{Preliminaries and the equal-degree input}
\label{sec:prelim}

\subsection{Valuative Newton--Okounkov bodies}

Let $Z$ be a smooth projective $r$-fold. An \emph{admissible flag} is a
chain $Z=Y_0\supset Y_1\supset\cdots\supset Y_r=\{z\}$, where each
$Y_k$ is an irreducible closed subvariety of codimension $k$, smooth at $z$.
For a nonzero section $s$ of a line bundle, $\nu_{Y_\bullet}(s)$ records
successive vanishing orders: at each step divide by a local equation
of $Y_k$ in $Y_{k-1}$ to the recorded order and restrict to $Y_k$; see
\cite[Section~1.1]{LazarsfeldMustata}. For a line bundle $B$ with a
nonzero section of some positive tensor power, write $mB:=B^{\otimes m}$
and define
\begin{equation}
 \label{eq:valuative}
 \Dval_{Y_\bullet}(B)
 =\overline{\conv}\left\{
 \frac{\nu_{Y_\bullet}(s)}m:
 m\geq1,\ 0\neq s\in H^0(Z,mB)\right\}.
\end{equation}
We use tensor notation for line bundles
and additive notation for divisor classes and intersection products.
We use the term \emph{valuative Newton--Okounkov body} also when $B$
is not big, following \cite{CHPW}. For big $B$ it is the usual
Newton--Okounkov body, denoted simply by $\Delta_{Y_\bullet}(B)$.

Write $\vol_r$ for Euclidean $r$-dimensional volume and
$\vol_Z(B)$ for the volume of a line bundle $B$ on $Z$. For big line
bundles, numerical invariance, homogeneity, and the volume theorem give
\begin{equation}
 \label{eq:volume-input}
 \vol_r\bigl(\Delta_{Y_\bullet}(B)\bigr)
 =\frac{\vol_Z(B)}{r!};
\end{equation}
see \cite{LazarsfeldMustata,KavehKhovanskii}. Numerical invariance is used
below only for big line bundles on the subproducts themselves, before
pullback. It is not asserted for arbitrary non-big valuative bodies.

Multiplication of sections gives the standard Minkowski inclusion
\begin{equation}
 \label{eq:minkowski}
 \Dval_{Y_\bullet}(B_1)+\cdots+\Dval_{Y_\bullet}(B_q)
 \subseteq\Dval_{Y_\bullet}(B_1\otimes\cdots\otimes B_q)
\end{equation}
whenever the bundles have sections in positive powers and the same flag
$Y_\bullet$ is used. This follows directly from \eqref{eq:valuative}, without a bigness
assumption: for $a=\nu(s)/m$ and $b=\nu(t)/r$, the section $s^rt^m$
has normalized valuation $a+b$. Taking closed convex hulls and iterating
proves \eqref{eq:minkowski}.

\subsection{Infinitesimal flags}

Recall that a \emph{complete linear flag} in an $r$-dimensional vector space $V$ is
a chain $0=W_0\subset\cdots\subset W_r=V$ with $\dim W_k=k$.
Here $\PP(V)$ denotes the space of one-dimensional subspaces of $V$;
projectivizing the nonzero terms gives the corresponding projective flag.

Suppose $n\geq2$, fix $x\in X$, and let
$\pi:\widetilde X=\operatorname{Bl}_xX\to X$ with exceptional divisor
$E\simeq\PP(T_xX)$. An \emph{infinitesimal admissible flag} is
\[
 \widetilde X=Y_0\supset Y_1=E\supset Y_2\supset\cdots
 \supset Y_n=\{y\},
\]
where $Y_k=\PP(W_{n-k+1})$ for $1\leq k\leq n$, and $W_\bullet$
is a complete linear flag in $T_xX$. Thus these flags are parametrized by
$\operatorname{Flag}(T_xX)$ introduced above. The first valuation coordinate
is the order of vanishing at $x$.

To relate the flag to a line bundle $B$, choose local parameters at $x$
and a local trivialization of $B$. A section $0\ne s\in H^0(X,mB)$ then
has a local representative $f=F_t+\text{higher-order terms}$, where
$t=\operatorname{ord}_x(s)$ and $F_t$ is a homogeneous polynomial of
degree $t$ on $T_xX$. After removing the order-$t$ vanishing along $E$,
the restriction of $\pi^*s$ to $E$ is represented by $F_t$; its orders
along the linear flag give the remaining valuation coordinates
\cite[Section~4.1, before Definition~4.1]{FLMinima}. Thus $B$ determines
which initial forms arise from global sections, and the flag determines
how they are valued.

For any big line bundle $B$ on $X$, generic constancy in families
\cite[Theorem~5.1]{LazarsfeldMustata} makes
$\Delta_{Y_\bullet}(\pi^*B)$ constant on a very general locus of linear
flags. This means that one excludes countably many proper Zariski-closed
subsets of the parameter space. The locus need not be Zariski open.
We denote the resulting body by $\Delta_x(B)$; the point $x$ remains fixed.
Thus $\Delta_x(B)$ denotes this common convex body, while $Y_\bullet$
denotes an actual flag. The equality
$\Delta_{Y_\bullet}(\pi^*B)=\Delta_x(B)$ means that the chosen flag
computes the generic value. See also \cite{FLCurves,KuronyaLozovanu2017}.
In expressions for infinitesimal valuations and bodies we suppress
$\pi^*$ when this causes no ambiguity. In dimension one we use the
ordinary point valuation, so $\Delta_x(L)=[0,d_1]$.

The external tensor product $L$ is ample, and therefore
\begin{equation}
 \label{eq:bodyvolume}
 \vol_n(\Delta_x(L))=\frac{L^n}{n!}=\prod_{i=1}^n d_i.
\end{equation}
For line bundles $M_i$ on $C_i$ and a point
$z\in\prod_{i=1}^j C_i$, the equal-degree input is
\begin{equation}
 \label{eq:equal-degree}
 \Delta_z(M_1\boxtimes\cdots\boxtimes M_j)=aS_j^{(j)}
 \quad\text{if }\deg M_1=\cdots=\deg M_j=a>0.
\end{equation}
For $a=1$ this is \cite[Theorem~1.1]{FLCurves}; numerical invariance and
homogeneity give the stated form for $a>0$.

\section{The main Minkowski formula}
\label{sec:proof}

\subsection{Simultaneous flags and infinitesimal pullbacks}

Let $V$ be an $n$-dimensional complex vector space. For an integer
$q\geq0$, let $V_1,\ldots,V_q$ be complex vector spaces and fix
surjective linear maps
\[
 \rho_\alpha:V\longrightarrow V_\alpha,\qquad
 r_\alpha=\dim V_\alpha\in\{1,\ldots,n\},\qquad 1\leq\alpha\leq q.
\]
Thus $\alpha$ labels the maps and their target spaces; $q$ counts the
maps and need not be bounded by $n$. Fix countably many dense
Zariski-open conditions on complete flags in $V$ and in each $V_\alpha$.
Satisfying all these conditions means lying in a chosen very general
locus, as in \cref{sec:prelim}. The conditions are fixed before the flags
are chosen; in \cref{cor:commonflag} they come from generic constancy of
the bodies.

\begin{lemma}[Simultaneous flags for linear quotients]
\label{lem:commonflag}
There is a complete flag $W_\bullet$ in $V$ satisfying the chosen flag
conditions such that, for every $\alpha$,
\[
 \rho_\alpha|_{W_{r_\alpha}}:W_{r_\alpha}\longrightarrow V_\alpha
 \quad\text{is an isomorphism},
\]
and the induced complete flag
$\rho_\alpha(W_1)\subset\cdots\subset\rho_\alpha(W_{r_\alpha})=V_\alpha$
satisfies the conditions chosen for flags in $V_\alpha$.
\end{lemma}

\begin{proof}
We will find one ordered basis of $V$ satisfying all requirements, then
take its successive spans.

To parametrize the choices, fix reference bases in $V$ and each
$V_\alpha$. Give the Cartesian product
$V^n=V\times\cdots\times V\simeq\CC^{n^2}$ its Zariski topology.
The ordered bases of $V$ form the nonempty open set
\[
 U=\{(v_1,\ldots,v_n)\in V^n:\det[v_1\ \cdots\ v_n]\ne0\},
\]
where brackets denote matrices of coordinate columns in the reference
bases. Since $V^n$ is irreducible, so is $U$.

For each $\alpha$, the image vectors form a basis precisely when
\[
 \det[\rho_\alpha(v_1)\ \cdots\ \rho_\alpha(v_{r_\alpha})]\ne0.
\]
This $r_\alpha\times r_\alpha$ determinant is a polynomial in the
coordinates of the $v_i$, so its zero set is closed and its complement
is open in $U$. On this open set, the map to ordered bases of $V_\alpha$
is surjective: lift any target basis and extend its independent lifts
to a basis of $V$. In particular, the open set is nonempty.

To impose the flag conditions, taking successive spans is algebraic,
and every complete flag has an adapted basis: a basis whose first $k$
vectors span its $k$-dimensional term. The preceding lifting argument
therefore makes each required source or target flag condition pull back
to a nonempty open subset of $U$, which is dense by irreducibility.

Finally, these countably many dense open conditions have a common point
by \cite[Lemma~10]{BanerjeeGuletskii}, since $\CC$ is uncountable.
For a basis at that point, set $W_k=\operatorname{span}(v_1,\ldots,v_k)$,
with $W_0=0$. Each restriction sends a basis of $W_{r_\alpha}$ to a basis
of $V_\alpha$, so it is an isomorphism; the source and image flags satisfy
all the chosen conditions.
\end{proof}

For the geometric application, put
\[
 X_{\leq j}=C_1\times\cdots\times C_j,\qquad
 x_{\leq j}=(x_1,\ldots,x_j),\qquad
 p_j:X\longrightarrow X_{\leq j}.
\]
Here $1\leq j\leq n$, and $p_j$ is the projection onto the first $j$
factors. Write
$\rho_j=d_xp_j:T_xX\to T_{x_{\leq j}}X_{\leq j}$ for its surjective
tangent map. Under $T_xX=\bigoplus_{i=1}^n T_{x_i}C_i$, this is the
projection onto the first $j$ summands. Fix an index set $J\subseteq\{1,\ldots,n\}$, a big line
bundle $A_j$ on $X_{\leq j}$ for each $j\in J$, and a finite family
$\mathcal B$ of big line bundles on $X$. Either family may be empty.
We require genericity for every $B\in\mathcal B$ on $X$ and for $A_j$
on each $X_{\leq j}$. Define the linear embedding
\[
 \iota_j:\RR^j\longrightarrow\RR^n,\qquad
 (t,u_2,\ldots,u_j)\longmapsto
 (t,\underbrace{0,\ldots,0}_{n-j},u_2,\ldots,u_j).
\]
For $j=1$, this means $\iota_1(t)=(t,0,\ldots,0)$.

\begin{corollary}[Compatible infinitesimal pullbacks]
\label{cor:commonflag}
For these data, there exists a complete linear flag $W_\bullet$ in
$T_xX$ such that
\[
 \rho_j|_{W_j}:W_j\xrightarrow{\ \sim\ }T_{x_{\leq j}}X_{\leq j}
 \qquad(j\in J),
\]
and the associated infinitesimal flag $Y_\bullet$ over $x$ and the flags
$Y_\bullet^{(j)}$ over $x_{\leq j}$ induced by
$\rho_j(W_1)\subset\cdots\subset\rho_j(W_j)$ satisfy
\[
 \Delta_{Y_\bullet}(B)=\Delta_x(B)
 \quad\text{for every }B\in\mathcal B,
\]
and, for every $j\in J$,
\begin{equation}
 \label{eq:pullback-body}
 \Dval_{Y_\bullet}(p_j^*A_j)
 =\iota_j\bigl(\Delta_{Y_\bullet^{(j)}}(A_j)\bigr)
 =\iota_j\bigl(\Delta_{x_{\leq j}}(A_j)\bigr).
\end{equation}
In dimension one the corresponding flag is the point valuation.
\end{corollary}

\begin{proof}
We choose compatible generic flags, compare valuations on the blow-ups,
and use K\"unneth to identify all sections.

For the flag choice, apply \cref{lem:commonflag} with source $T_xX$,
targets $T_{x_{\leq j}}X_{\leq j}$, and maps $\rho_j$ for $j\in J$
($q=|J|$, with target dimension $j$). Generic constancy
\cite[Theorem~5.1]{LazarsfeldMustata} supplies the source flag conditions
from the bundles $B\in\mathcal B$ and the target flag conditions from
the $A_j$. The lemma makes these conditions and the stated tangent-space
isomorphisms hold simultaneously.

For the valuation comparison, assume $n\geq2$, the curve case being
immediate. Let $F_t$ be the initial homogeneous term of a local
representative of $0\ne s\in H^0(X_{\leq j},mA_j)$, with
$t=\operatorname{ord}_{x_{\leq j}}(s)$. On $\widetilde X$, the section
$\pi^*p_j^*s$ has order $t$ along $E$; after removing this vanishing,
its restriction to $E$ is represented by $F_t\circ\rho_j$
\cite[Section~4.1, before Definition~4.1]{FLMinima}. This polynomial is
not identically zero on $W_j$ because $\rho_j|_{W_j}$ is an isomorphism.
The successive restrictions from $\PP(W_n)$ down to $\PP(W_j)$ therefore
contribute $n-j$ zeros, and the remaining flag is identified with the
target flag by this isomorphism. Thus, writing
$\nu_{Y_\bullet^{(j)}}(s)=(t,u_2,\ldots,u_j)$, we obtain
\begin{equation}
 \label{eq:pullback-values}
 \nu_{Y_\bullet}(p_j^*s)
 =(t,\underbrace{0,\ldots,0}_{n-j},u_2,\ldots,u_j)
 =\iota_j\bigl(\nu_{Y_\bullet^{(j)}}(s)\bigr).
\end{equation}
The inserted zeros occupy coordinates $2,\ldots,n-j+1$; the first
coordinate remains $t$.

Finally, $p_j^*A_j$ is the external tensor product of $A_j$ with the
trivial bundles on the remaining curves. K\"unneth in cohomological
degree zero \cite[Tag~0BED]{Stacks} and $H^0(C_i,\cO_{C_i})=\CC$ give,
for every $m\geq1$,
\[
 \begin{aligned}
 H^0(X,mp_j^*A_j)
 &\simeq H^0(X_{\leq j},mA_j)\otimes_{\CC}
 \bigotimes_{i=j+1}^n H^0(C_i,\cO_{C_i})\\
 &\simeq H^0(X_{\leq j},mA_j).
 \end{aligned}
\]
Here $s\otimes1\otimes\cdots\otimes1$ corresponds to $p_j^*s$, so
\eqref{eq:pullback-values} covers all sections upstairs. Normalizing and
taking closed convex hulls proves the first equality in
\eqref{eq:pullback-body}; the chosen genericity gives the remaining
assertions.
\end{proof}

\begin{lemma}[Simultaneous generic relabeling]
\label{lem:simultaneous-relabeling}
Let $B^{(1)},\ldots,B^{(q)}$ be positive-degree external products on $X$.
Use the factor-permutation notation introduced before \cref{thm:bodyformula}.
For a section $s$ of a line bundle $B$, write
$s^\sigma:=(\tau_\sigma^{-1})^*s$, and let
\[
 \Phi_\sigma:\operatorname{Flag}(T_xX)\longrightarrow
 \operatorname{Flag}(T_{x^\sigma}X^\sigma)
\]
be the isomorphism induced by the differential of $\tau_\sigma$.
For each $\alpha$, choose
$\sigma_\alpha\in\mathfrak S_n$ and set
$B^{(\alpha),\sigma_\alpha}:=(B^{(\alpha)})^{\sigma_\alpha}$. Then there
is a very general, hence nonempty, locus $G\subset\operatorname{Flag}(T_xX)$ such that for
every $Y_\bullet\in G$ and every $\alpha$,
\begin{equation}
 \label{eq:simultaneous-relabeling}
 \Delta_{Y_\bullet}(B^{(\alpha)})
 =\Delta_x(B^{(\alpha)})
 =\Delta_{Y_\bullet^{\sigma_\alpha}}
   (B^{(\alpha),\sigma_\alpha})
 =\Delta_{x^{\sigma_\alpha}}
   (B^{(\alpha),\sigma_\alpha}).
\end{equation}
In particular, the different bundles may be reordered independently while
being represented by one common flag $Y_\bullet$ on the original product.
\end{lemma}

\begin{proof}
For each $\alpha$, choose a very general locus $G_\alpha$ on which
$\Delta_{Y_\bullet}(B^{(\alpha)})=\Delta_x(B^{(\alpha)})$, and a very
general locus $G_\alpha^{\sigma_\alpha}$ on the reordered flag variety on
which the generic body of $B^{(\alpha),\sigma_\alpha}$ is computed.
Since $\Phi_{\sigma_\alpha}$ is an isomorphism,
$\Phi_{\sigma_\alpha}^{-1}(G_\alpha^{\sigma_\alpha})$ is again very
general. Hence
\[
 G:=\bigcap_{\alpha=1}^q
 \Bigl(G_\alpha\cap
 \Phi_{\sigma_\alpha}^{-1}(G_\alpha^{\sigma_\alpha})\Bigr)
\]
is very general and nonempty: its complement is a countable union of proper
Zariski-closed subsets of the irreducible flag variety, so the uncountability
argument used in \cref{lem:commonflag} applies.

For $Y_\bullet\in G$, the factor-permutation isomorphism transports both the
bundle and the flag. Successive orders of vanishing are invariant under this
simultaneous transport, so
\[
 \nu_{Y_\bullet}(s)
 =\nu_{Y_\bullet^{\sigma_\alpha}}(s^{\sigma_\alpha}),
 \qquad
 \Delta_{Y_\bullet}(B^{(\alpha)})
 =\Delta_{Y_\bullet^{\sigma_\alpha}}
   (B^{(\alpha),\sigma_\alpha}).
\]
Together with the defining properties of $G_\alpha$ and
$G_\alpha^{\sigma_\alpha}$ this gives \eqref{eq:simultaneous-relabeling}.
\end{proof}

\subsection{A mixed-volume identity}

The volume comparison in the main proof reduces to \cref{lem:volume}.
We use normalized mixed volume, $\MV(K,\ldots,K)=n!\vol_n(K)$.
Equivalently, $\MV(K_1,\ldots,K_n)$ is the coefficient of
$t_1\cdots t_n$ in $\vol_n(\sum_{i=1}^n t_iK_i)$.
We use its standard multilinearity, monotonicity, and continuity
\cite[Chapter~5]{Schneider}. For two distinct coordinate simplices
the mixed-volume formula is \cite[Corollary~4.3]{Sadovsky}; the argument
below uses assignment duality to obtain the arbitrary-tuple form
needed here.

For an $n\times n$ matrix $(c_{ri})$ of nonnegative real numbers, set
$K_r=\conv\{0,c_{r1}e_1,\ldots,c_{rn}e_n\}$ for $1\leq r\leq n$.
\begin{lemma}[Coordinate-simplex assignment formula]
\label{lem:assignment}
The normalized mixed volume of these simplices is
\begin{equation}
 \label{eq:assignment}
 \MV(K_1,\ldots,K_n)
 =\max_{\sigma\in\mathfrak S_n}\prod_{r=1}^n c_{r,\sigma(r)}.
\end{equation}
\end{lemma}

\begin{proof}
We match lower and upper bounds by monotonicity and assignment duality.
Put $M=\max_\sigma\prod_r c_{r,\sigma(r)}$. The segments
$[0,c_{r,\sigma(r)}e_{\sigma(r)}]\subseteq K_r$ give
$\MV(K_1,\ldots,K_n)\geq M$.

For the upper bound, assume first that $c_{ri}>0$. Assignment duality
applied to $(\log c_{ri})$
\cite[Corollary~2.6a and Theorem~3.7]{Schrijver} gives positive numbers
$a_r,b_i$ such that
\[
 c_{ri}\leq a_rb_i,\qquad \prod_r a_r\prod_i b_i=M.
\]
Set $R=\conv\{0,b_1e_1,\ldots,b_ne_n\}$. Then $K_r\subseteq a_rR$, so
\[
 \MV(K_1,\ldots,K_n)
 \leq\Bigl(\prod_r a_r\Bigr)n!\vol_n(R)
 =\prod_r a_r\prod_i b_i=M.
\]
For zero entries, replace $c_{ri}$ by $c_{ri}+\varepsilon$, apply the
positive case, and let $\varepsilon\to0^+$; continuity then proves the
formula.
\end{proof}

For $1\leq j\leq n$, put
$D_j=\conv\{0,je_1,(j-1)e_2,\ldots,e_j\}\subset\RR^n$.

\begin{lemma}[Volume of the layered sum]
\label{lem:volume}
For any nonnegative real coefficients $\lambda_1,\ldots,\lambda_n$,
\begin{equation}
 \label{eq:layer-volume}
 \vol_n\left(\sum_{j=1}^n\lambda_jD_j\right)
 =\prod_{i=1}^n\left(\sum_{j=i}^n\lambda_j\right).
\end{equation}
\end{lemma}

\begin{proof}
We use the standard mixed-volume expansion \cite[Chapter~5]{Schneider}
and compare coefficients. For $1\leq\kappa_1\leq\cdots\leq\kappa_n\leq n$,
\cref{lem:assignment} applied to $c_{ri}=\max\{\kappa_r-i+1,0\}$ gives
\begin{equation}
 \label{eq:layer-mv}
 \MV(D_{\kappa_1},\ldots,D_{\kappa_n})
 =\prod_{r=1}^n\max\{\kappa_r-r+1,0\}.
\end{equation}
Indeed, inverted positive assignments can be
interchanged without decreasing their product, since for $r<s$ and $i<j$,
\[
 (\kappa_r-i+1)(\kappa_s-j+1)
 -(\kappa_r-j+1)(\kappa_s-i+1)
 =(\kappa_s-\kappa_r)(j-i)\geq0.
\]
Thus the identity assignment is optimal. If $\kappa_r<r$ for some $r$,
the first $r$ rows have fewer than $r$ available columns, and both sides
of \eqref{eq:layer-mv} vanish.

To compare coefficients, let $m_j$ count the occurrences of $j$ among
the $\kappa_r$. The coefficient of $\prod_j\lambda_j^{m_j}$ on the left
of \eqref{eq:layer-volume} is
\[
 \frac{\prod_{r=1}^n\max\{\kappa_r-r+1,0\}}{\prod_jm_j!}.
\]
The same coefficient occurs on the right: assign the labelled indices
$\kappa_r$, in increasing order, to distinct factors numbered at most
$\kappa_r$. There are $\max\{\kappa_r-r+1,0\}$ choices at step $r$; division
by $\prod_jm_j!$ forgets labels on equal indices. Hence all coefficients
agree, proving \eqref{eq:layer-volume}.
\end{proof}

\subsection{Completion of the proof}

\begin{proof}[Proof of \cref{thm:bodyformula}]
We first assume $d=d^\downarrow$, equivalently
$d_1\geq\cdots\geq d_n$. We decompose $L$ into equal-degree external
products on nested subproducts. \Cref{cor:commonflag} places the bodies of their pullbacks
in a common valuation space, giving the required Minkowski inclusion.
The volume identity in \cref{lem:volume} then forces equality.
The case $n=1$ is the point-valuation formula $[0,d_1]$; assume
$n\geq2$ for the argument below.

To construct the decomposition, set $d_{n+1}=0$ and
$\lambda_j=d_j-d_{j+1}$ for $1\leq j\leq n$.
For $j<n$ and $i\leq j$, choose a line bundle $M_{j,i}$ on $C_i$ of
degree $\lambda_j$ if $\lambda_j>0$, and set $M_{j,i}=\cO_{C_i}$ if
$\lambda_j=0$. Define the remaining factor bundles and the bundles
on subproducts by
\[
 M_{n,i}=L_i\otimes
 \left(\bigotimes_{j=i}^{n-1}M_{j,i}\right)^{-1},\qquad
 A_j=\boxtimes_{i=1}^jM_{j,i}\quad\text{on }X_{\leq j}.
\]
An empty tensor product is the trivial bundle. For each $i$,
$\deg M_{n,i}=d_i-\sum_{j=i}^{n-1}\lambda_j=d_n=\lambda_n>0$.
Consequently the chosen bundles satisfy the actual tensor identity
\begin{equation}
 \label{eq:layer-decomposition}
 L\simeq\bigotimes_{j=1}^n p_j^*A_j.
\end{equation}
For $\lambda_j>0$, the bundle $A_j$ is ample with every factor degree
equal to $\lambda_j$; for $\lambda_j=0$, it is trivial.

To obtain the inclusion, apply \cref{cor:commonflag} with
$J=\{j:\lambda_j>0\}$, the bundles $A_j$ just constructed, and
$\mathcal B=\{L\}$. For the resulting flag, \eqref{eq:equal-degree} gives
\[
 \Dval_{Y_\bullet}(p_j^*A_j)
 =\iota_j\bigl(\lambda_jS_j^{(j)}\bigr)
 =\lambda_jS_j^{(n)}\qquad(j\in J).
\]
The tensor decomposition identifies products of pullback sections by
\[
 \bigotimes_{j\in J}H^0(X_{\leq j},mA_j)\longrightarrow H^0(X,mL),
 \qquad\bigotimes_{j\in J}s_j\longmapsto\prod_{j\in J}p_j^*s_j.
\]
Valuations add under this multiplication. Thus, omitting the trivial
factors and applying \eqref{eq:minkowski}, we obtain
\begin{equation}
 \label{eq:main-inclusion}
 \mathcal P(d)=\sum_{j=1}^n\lambda_jS_j^{(n)}\subseteq\Delta_x(L).
\end{equation}

To compare volumes, use the unimodular linear map
\[
 T(v_1,\ldots,v_n)
 =\left(v_1-\sum_{r=2}^nv_r,v_n,v_{n-1},\ldots,v_2\right).
\]
It sends $S_j^{(n)}$ to $D_j$. Since
$\sum_{j=i}^n\lambda_j=d_i$, \cref{lem:volume} and
\eqref{eq:bodyvolume} give
\[
 \vol_n(\mathcal P(d))=\prod_{i=1}^n d_i=\vol_n(\Delta_x(L)).
\]

Finally, $\mathcal P(d)$ is full dimensional because $\lambda_n=d_n>0$.
A proper containment of compact convex sets containing a full-dimensional
convex body strictly increases volume. Hence \eqref{eq:main-inclusion}
and the volume equality force $\mathcal P(d)=\Delta_x(L)$ in the
decreasing case.

We now remove this assumption and prove the simultaneous assertion.
Consider a finite family and independently chosen sorting permutations as in
the theorem. Apply \cref{lem:simultaneous-relabeling}. It gives a very general,
nonempty locus $G$ on the original flag variety. For every $Y_\bullet\in G$
and each $\alpha$, the transported flag $Y_\bullet^{\sigma_\alpha}$ is generic
for the corresponding reordered bundle. That bundle has decreasing degree
vector $d^{(\alpha)\downarrow}$, so the decreasing-order case just proved on
$X^{\sigma_\alpha}$ gives
\[
 \Delta_{Y_\bullet^{\sigma_\alpha}}
   (L^{(\alpha),\sigma_\alpha})
 =\Delta_{x^{\sigma_\alpha}}
   (L^{(\alpha),\sigma_\alpha})
 =\mathcal P(d^{(\alpha)\downarrow})
 =\mathcal P(d^{(\alpha)}).
\]
Together with \eqref{eq:simultaneous-relabeling}, this proves
\eqref{eq:main-simultaneous} for all $\alpha$ at once. Taking $q=1$ removes
the temporary decreasing-order assumption and yields \eqref{eq:main} for an
arbitrary original factor order.
\end{proof}

When all $d_i=d$, only the last summand remains and
$\Delta_x(L)=dS_n^{(n)}$. For $n=2,3$, expansion of \eqref{eq:main}
recovers \cite[Theorem~1.2]{FLCurves}.

\section{Equality in Minkowski inclusion}
\label{sec:additivity}

The proof below uses only the following elementary injectivity property of the
explicit body map.

\begin{remark}[Injectivity of the explicit body map]
\label{rem:P-injective}
For positive vectors $b,c\in\RR_{>0}^n$,
\[
 \mathcal P(b)=\mathcal P(c)
 \quad\Longleftrightarrow\quad
 b^\downarrow=c^\downarrow.
\]
Indeed, set
\[
 \Psi(v_1,\ldots,v_n)=(v_2,\ldots,v_n,v_1-v_2-\cdots-v_n).
\]
For $K=\Psi(\mathcal P(b))$, let
$a_r(K)=\max\{t\geq0:te_r\in K\}$ and $a_0(K)=0$. Since
\[
 \Psi(S_j^{(n)})
 =\conv\{0,e_{n-j+1},2e_{n-j+2},\ldots,je_n\}
\]
and all summands lie in the nonnegative orthant, their coordinate intercepts
add. Hence
\[
 a_r(K)=\sum_{q=n-r+1}^n b_q^\downarrow,
 \qquad
 b_i^\downarrow=a_{n-i+1}(K)-a_{n-i}(K).
\]
Thus $\mathcal P(b)$ determines $b^\downarrow$ uniquely; the converse is
immediate from \eqref{eq:Pdef}.
\end{remark}

Take $L=\boxtimes_iL_i$ and $M=\boxtimes_iM_i$ on the same product, and
write
\[
 d=(d_i)_i=(\deg L_i)_i,\qquad
 d'=(d_i')_i=(\deg M_i)_i,
\]
with all entries positive.

\begin{corollary}[Minkowski additivity]
\label{cor:additivity}
With this notation, there exists a very general, hence nonempty, locus
$G\subset\operatorname{Flag}(T_xX)$ simultaneously generic for
$L$, $M$, and $L\otimes M$. For every $Y_\bullet\in G$,
\[
 \Delta_{Y_\bullet}(L)+\Delta_{Y_\bullet}(M)
 \subseteq\Delta_{Y_\bullet}(L\otimes M),
\]
and equality holds if and only if the two degree vectors admit a common
decreasing factor order. Equivalently, this occurs if and only if
\begin{equation}
 \label{eq:common-order}
 (d_i-d_j)(d_i'-d_j')\geq0\quad\text{for all }i,j.
\end{equation}
On $G$ the three bodies equal their generic values, so the same condition is
also equivalent to
\begin{equation}
 \label{eq:additivity}
 \Delta_x(L)+\Delta_x(M)=\Delta_x(L\otimes M).
\end{equation}
\end{corollary}

\begin{proof}
Apply the simultaneous clause of \cref{thm:bodyformula} to
$L$, $M$, and $L\otimes M$, choosing independently permutations which put
$d$, $d'$, and $d+d'$ in decreasing order. This gives a very general,
nonempty locus $G$ as above. For every $Y_\bullet\in G$,
\[
 \Delta_{Y_\bullet}(L)=\Delta_x(L)=\mathcal P(d^\downarrow),\qquad
 \Delta_{Y_\bullet}(M)=\Delta_x(M)=\mathcal P((d')^\downarrow),
\]
\[
 \Delta_{Y_\bullet}(L\otimes M)=\Delta_x(L\otimes M)
 =\mathcal P((d+d')^\downarrow)=\mathcal P(d+d').
\]
Multiplication of sections gives the Minkowski inclusion
in the common valuation coordinates of $Y_\bullet$.
Since $d^\downarrow$ and $(d')^\downarrow$ are decreasing, the definition
\eqref{eq:Pdef} is linear on the decreasing cone, and therefore
\[
 \Delta_x(L)+\Delta_x(M)
 =\mathcal P(d^\downarrow)+\mathcal P((d')^\downarrow)
 =\mathcal P(d^\downarrow+(d')^\downarrow).
\]
Thus equality in the Minkowski inclusion is equivalent to
\[
 \mathcal P(d^\downarrow+(d')^\downarrow)=\mathcal P(d+d').
\]
By Remark~\ref{rem:P-injective}, this holds exactly when
\begin{equation}
 \label{eq:sorted-sum}
 (d+d')^\downarrow=d^\downarrow+(d')^\downarrow.
\end{equation}

To characterize this identity, a common decreasing factor order plainly
suffices. Conversely, taking squared Euclidean norms in
\eqref{eq:sorted-sum} gives
\[
 \sum_i d_i d_i'=\sum_i d_i^\downarrow (d'^\downarrow)_i.
\]
The decreasing pairing maximizes this scalar product by the rearrangement
inequality. If $d_i>d_j$ but $d_i'<d_j'$, swapping these two $d'$-entries
increases the scalar product by
$(d_i-d_j)(d_j'-d_i')>0$, a contradiction. Hence \eqref{eq:common-order} holds, and sorting the ties gives a common
decreasing factor order. Thus equality of the bodies, the sorted-sum identity, and
a common decreasing factor order are equivalent, as claimed.
\end{proof}

For example, on a product of two curves, degree vectors $(3,1)$ and
$(1,3)$ give
\[
 \vol_2(\Delta_x(L)+\Delta_x(M))=6\cdot2=12,
 \qquad \vol_2(\Delta_x(L\otimes M))=4\cdot4=16.
\]
Thus Minkowski inclusion is strict even for ample external tensor products.
The common factor indexing in \cref{cor:additivity} matters: the two
individual bodies remember their sorted degrees, but not how their factors
are paired in the tensor product.

\end{document}